\documentclass[a4paper]{amsart}
\usepackage{amsthm,amsfonts,amsmath,amssymb}
\usepackage[abs]{overpic}
\usepackage{thm-restate}

\newtheorem{theorem}{Theorem}[section]
\newtheorem{proposition}[theorem]{Proposition}

\newtheorem{claim}[theorem]{Claim}

\newtheorem{example}[theorem]{Example}

\newtheorem{remark}[theorem]{Remark}
\newtheorem{problem}[theorem]{Problem}

\begin{document}
\title[MONTESINOS KNOTS WITH DELTA-UNKNOTTING NUMBER ONE]{MONTESINOS KNOTS WITH DELTA-UNKNOTTING NUMBER ONE}
\author{Kazumichi Nakamura}
\email{mi-ka-07130703@ezweb.ne.jp}





\begin{abstract}
The $\Delta$-unknotting number for a knot is defined as the minimum number of $\Delta$-moves needed to deform the knot into the trivial knot.
In this paper, we discuss Montesinos knots whose $\Delta$-unknotting number is equal to one.
We propose a conjectural characterization of Montesinos knots with $\Delta$-unknotting number one and provide examples admitting distinct $\Delta$-moves, each of which deforms the knot into the trivial knot.
Furthermore, we present examples illustrating the relationship between the locations of 
$\Delta$-moves and $X$-moves for knots with
$u^{\Delta}(K) = 1$, considering separately the cases $u(K) = 1$ and $u(K) = 2$.
\end{abstract}

\maketitle

\noindent\textbf{keywords:}$\Delta$-move, $\Delta$-unknotting number, $X$-move, unknotting number, Conway's normal form, two-bridge knot, pretzel knot, Montesinos knot.

\section{Introduction}
In this paper, we study Montesinos knots with $\Delta$-unknotting number one.

In \cite{MaN}, H. Murakami and Y. Nakanishi introduced a local move on regular diagrams of oriented knots and links, called a $\Delta$-move (or $\Delta$-unknotting operation), as illustrated in Figure \ref{fig:delta}.
\begin{figure}[htbp]
 \centering    \includegraphics[width=0.8\linewidth]
    {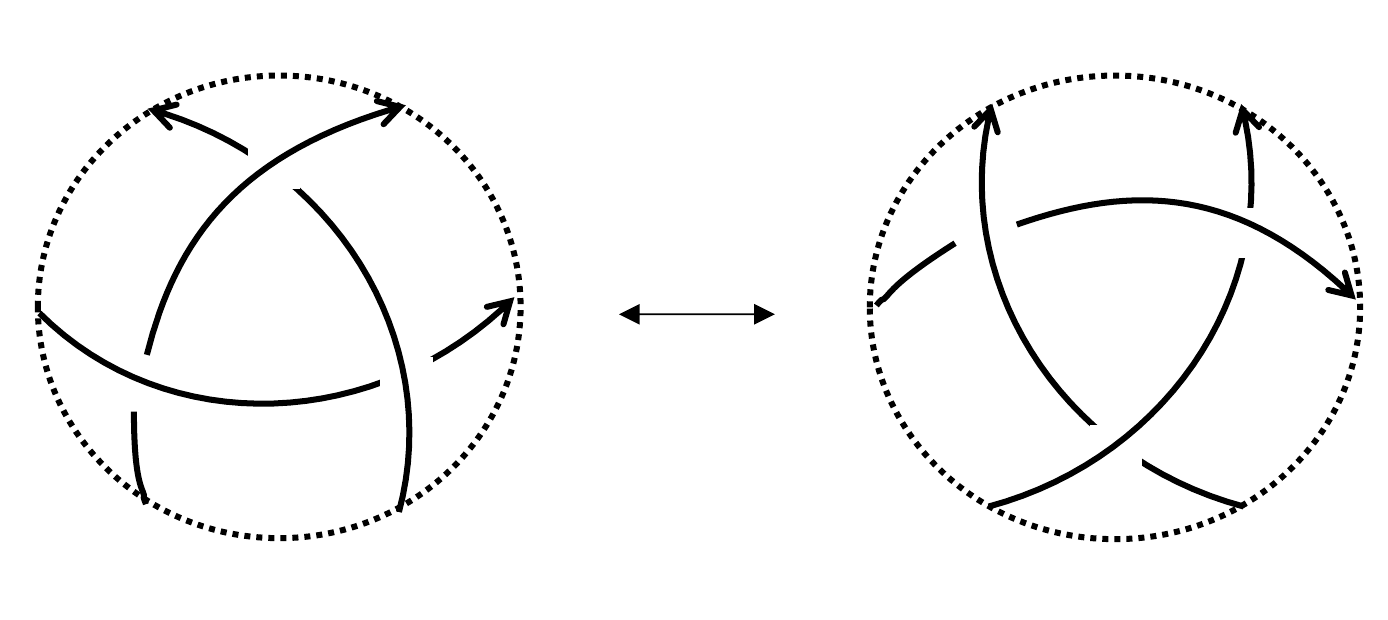}
 \caption{A $\Delta$-move.}
 \label{fig:delta}
\end{figure}
They proved that two knots can be deformed into each other by a finite sequence of $\Delta$-moves. 
The $\Delta$-Gordian distance $d_G^{\Delta}(K, K^{'})$ of two oriented knots $K$ and $K^{'}$ is defined as the minimum number of $\Delta$-moves needed to deform a diagram of $K$ into that of $K{'}$. 
The $\Delta$-unknotting number $u^{\Delta}(K)$  of an oriented knot $K$ is defined as the $\Delta$-Gordian distance $d_G^{\Delta}(K, O)$ of $K$ and the trivial knot $O$.

In \cite{Naka3}, we propose the conjecture and state the theorem concerning two-bridge knots with $u^{\Delta}(K) = 1$.
In \cite{Naka5}, we discuss ${\Delta}$-unknotting numbers for Montesinos knots.

In this paper, we propose the following conjecture and state the following theorem concerning Montesinos knots with $u^{\Delta}(K) = 1$.
The definition of a Type I rational tangle will be given in Section \ref{sec3}.

\begin{restatable}{conjecture}{conjmone} \label{thm:conjmone}
Let $K = M\big(0; (\alpha_1, \beta_1), \dots, (\alpha_r, \beta_r)\big)$
be a nontrivial Montesinos knot with $r \geq 3$. Then, the following two conditions are equivalent.
\begin{enumerate}
    \item       
    $u^{\Delta}(K) = 1$.
    \item       
    $K$ or its mirror image $K^*$ can be expressed as
    $M(0;(\alpha_1, \beta_1), (\alpha_2, \beta_2), (\alpha_3, \beta_3))$,
    where $T(\beta_3 / \alpha_3)$ is of Type I, 
    $|\alpha_1 \beta_2 + \beta_1 \alpha_2| = 1$, and
    $\alpha_1, \alpha_2, \alpha_3, \beta_1, \beta_2, \beta_3$ are non zero integers.
\end{enumerate}
\end{restatable}

Theorem \ref{thm:mone} establishes the implication (2) $\Rightarrow$ (1) in Conjecture \ref{thm:conjmone}.

\begin{restatable}{theorem}{mone} \label{thm:mone}
Let $K = M(0;(\alpha_1, \beta_1), (\alpha_2, \beta_2), (\alpha_3, \beta_3))$ be a nontrivial Montesinos knot,
where $T(\beta_3 / \alpha_3)$ is of Type I, 
$|\alpha_1 \beta_2 + \beta_1 \alpha_2| = 1$, and
$\alpha_1, \alpha_2, \alpha_3, \beta_1, \beta_2, \beta_3$ are non zero integers.
Then, we have $u^{\Delta}(K) = 1$.
\end{restatable}

Regarding the ordinary unknotting number for  Montesinos knots, we recall the following proposition from \cite{Tor}.

Let
$K$ be a  Montesinos knot with $r \geq 3$.
Then $u(K) = 1$ and the unknotting operation is realized in a standard diagram if and only if the following condition holds
(*) $K = M(0;(a, -c), (b, d), (2mn \pm 1, 2n^2))$, where 
$a, b, c, d, m, n$ are some non zero integers,
$m$ and $n$ are coprime, and $ad - cb = 1$.

\vspace{1em}
Theorem \ref{thm:moneone} follows from Theorem \ref{thm:mone}.

\begin{restatable}{theorem}{moneone} \label{thm:moneone}
Let $K = M(0;(\alpha_1, \beta_1), (\alpha_2, \beta_2), (\alpha_3, \beta_3))$ be a nontrivial Montesinos knot,
where $T(\beta_2 / \alpha_2)$ and $T(\beta_3 / \alpha_3)$ are of Type I, 
$|\alpha_1 \beta_2 + \beta_1 \alpha_2| = |\alpha_1 \beta_3 + \beta_1 \alpha_3| = 1$, and
$\alpha_1, \alpha_2, \alpha_3, \beta_1, \beta_2, \beta_3$ are non zero integers.
Then, we have $u^{\Delta}(K) = 1$.

In particular, performing a $\Delta$-move in one of $T(\beta_2 / \alpha_2)$ or $T(\beta_3 / \alpha_3)$ yields the trivial knot.
\end{restatable}

The following problem was posed by Y. Uchida. 

\begin{problem}\label{problemdelta}
Let $K$ be a knot with $u^{\Delta}(K)=1$.
Suppose that $\Delta^{'}$ and $\Delta^{''}$ are $\Delta$-moves that deform $K$ into the trivial knot.
Can $\Delta^{'}$ be transformed into $\Delta^{''}$ by a sequence of moves shown in Figure~\ref{fig:deltalr}?
Note that the ${\Delta}$-move on ${\Delta}_l$ is equivalent to the ${\Delta}$-move on ${\Delta}_r$.
\end{problem}
\begin{figure}[htbp]
    \centering
    \includegraphics[width=0.6\linewidth]{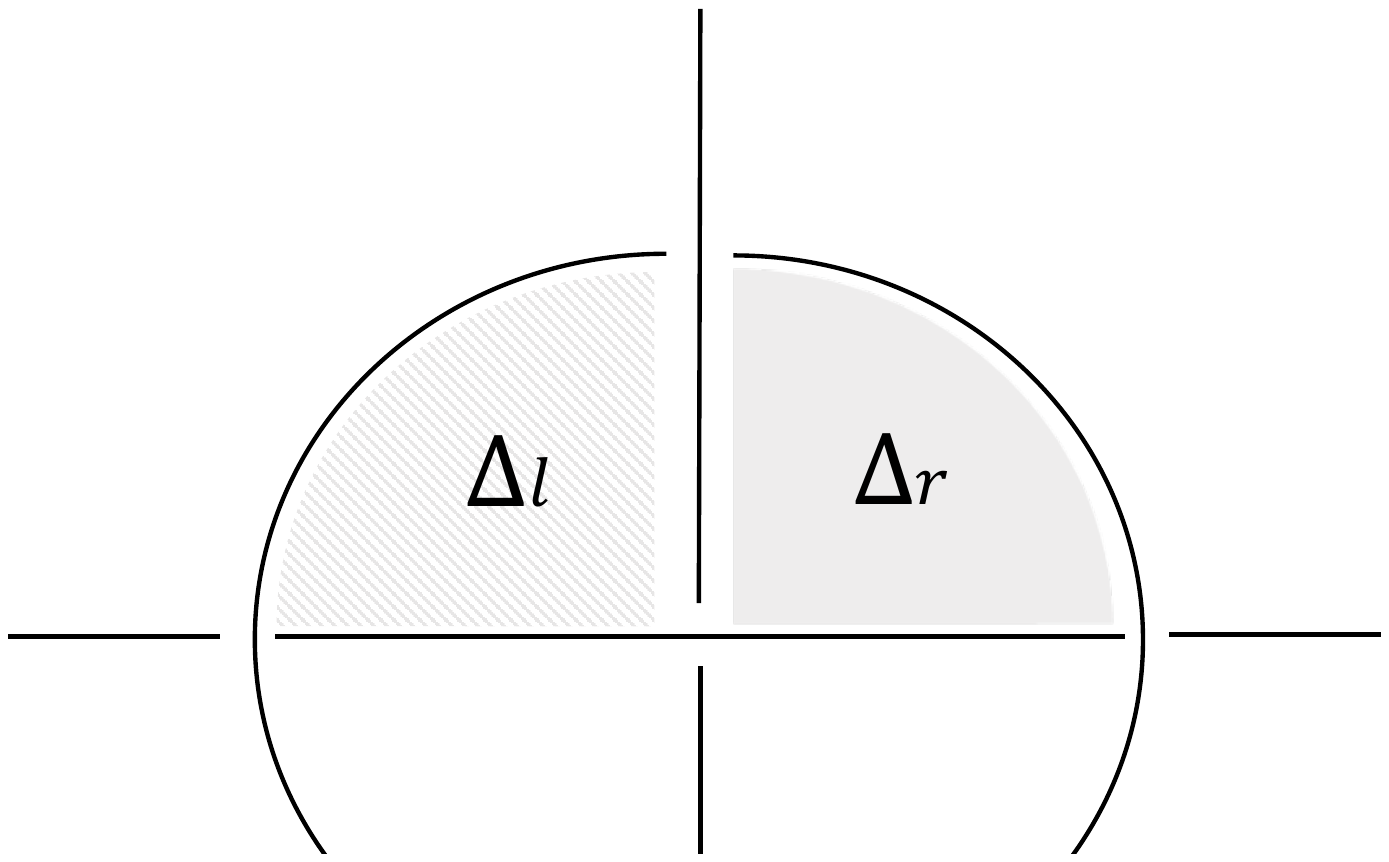}
    \caption{${\Delta}_r$ and ${\Delta}_l$.}
    \label{fig:deltalr}
\end{figure}

Montesinos knots $9_{28}$ and $10_{60}$ (see Figure \ref{fig:1060}) are known counterexamples to Problem \ref{problemdelta}.
Theorem~\ref{thm:moneone} provides a method for constructing infinitely many further counterexamples.
In Section~\ref{sec3}, we give explicit infinite families of counterexamples obtained from Theorem~\ref{thm:moneone},
including cases where 
$T(\beta_2 / \alpha_2) \cong T(\beta_3 / \alpha_3)$
and cases where
$T(\beta_2 / \alpha_2) \not\cong T(\beta_3 / \alpha_3)$
(Examples \ref{ex31}, \ref{ex32}, \ref{ex33}, \ref{ex34}, and \ref{ex35}).

\begin{figure}[htbp]
    \centering
    \includegraphics[width=0.8\linewidth]{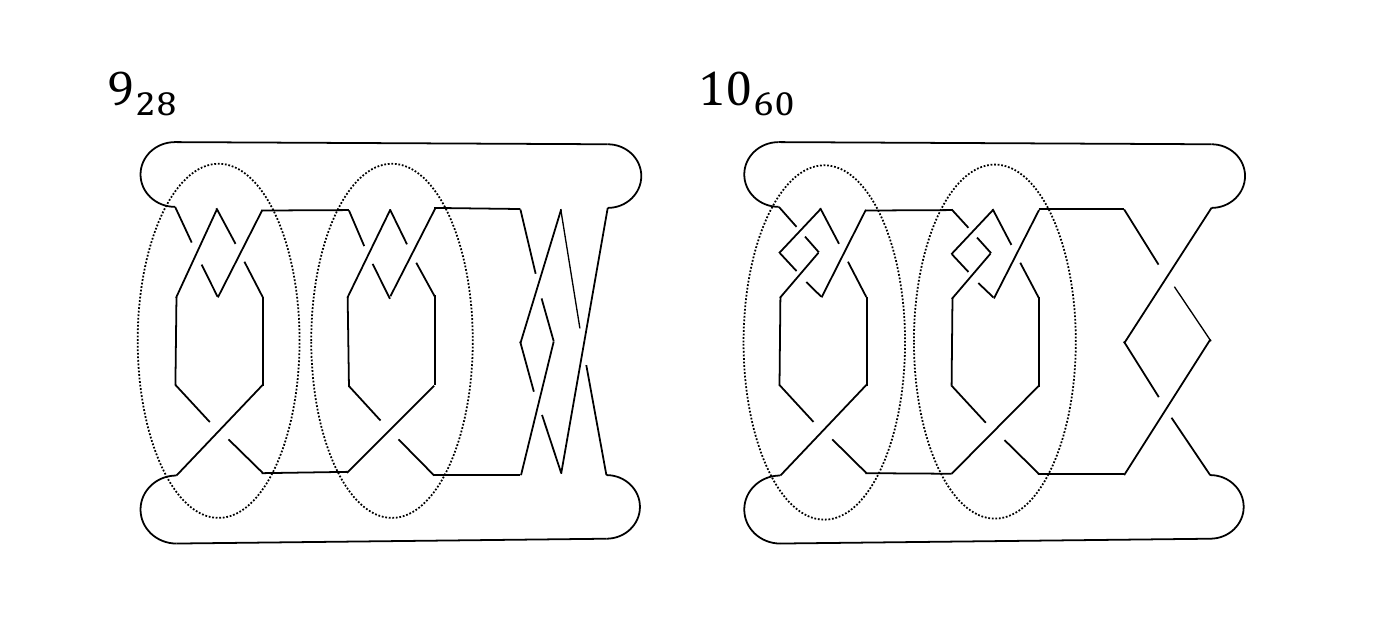}
    \caption{Diagrams of $9_{28}$ and $10_{60}$.}
    \label{fig:1060}
\end{figure}

\vspace{1em}

As a supplementary observation, we give an example in Section \ref{sec3} showing that the Conway polynomial alone cannot determine the $\Delta$-unknotting number
(Remark~\ref{remark:10141}).

\vspace{1em}

In \cite{Uchi}, the relationship between $\Delta$-moves and $X$-moves was discussed.
In Section \ref{sec4}, we present examples 
illustrating the relationship between the locations of $\Delta$-moves and $X$-moves for knots with
$u^{\Delta}(K) = 1$, considering separately the cases $u(K) = 1$ and $u(K) = 2$.

\section{Preliminaries}\label{sec2}

\subsection{Two-Bridge Knots with $\Delta$-Unknotting Number One} 
\label{subsec21}

\leavevmode

A knot $K$ is said to be a two-bridge knot if $K$ has a diagram as in Figure \ref{fig:twobridge}, called Conway's normal form. For a knot diagram as in Figure \ref{fig:twobridge}, each $|c_i|$ presents the number of half-twists for integers $c_1, c_2, ... , c_n$. In this paper, for the sign of $\alpha_i$, we assume that a right-handed half-twist is positive when $i$ is odd, and a left-handed half-twist is positive when $i$ is even. We denote this knot diagram by $C(c_1, c_2, ... , c_n)$.
\begin{figure}[htbp]
 \centering
 \includegraphics[width=1\linewidth]{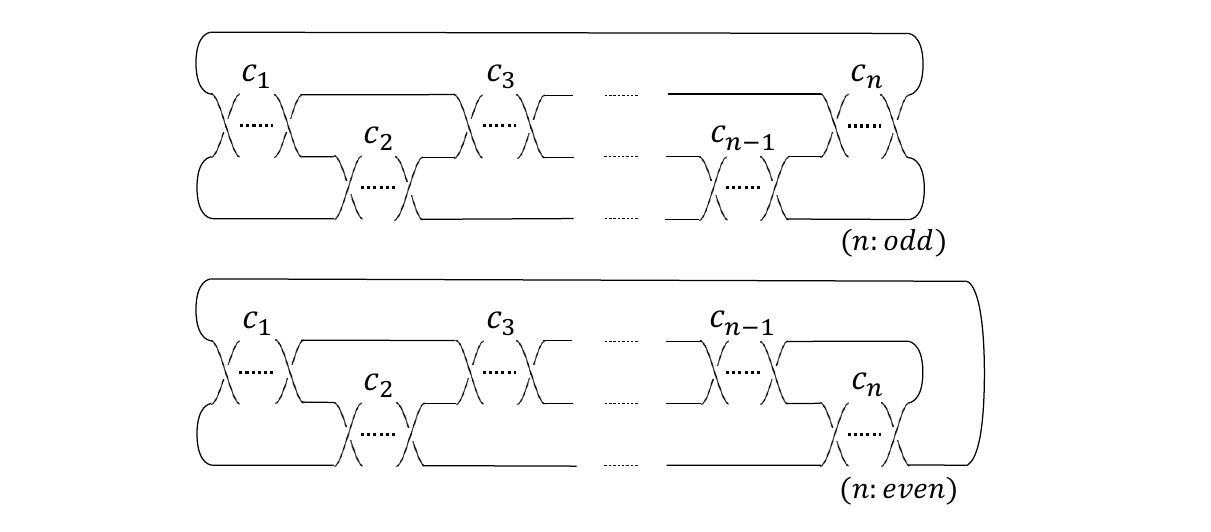}
 \caption{A two-bridge knot $C(c_1, c_2, ... , c_n)$.}
 \label{fig:twobridge}
\end{figure}

In \cite{Naka3}, we proposed the following conjecture and proved the following theorem for two-bridge knots. 

\begin{restatable}{conjecture}{conj1} (\cite{Naka3}) \label{thm:conj1}
Let $K$ be a nontrivial two-bridge knot. Then, the following two conditions are equivalent. 
\begin{enumerate}
    \item       
    $u^{\Delta}(K) = 1$.
    \item       
    $K$ or its mirror image $K^*$ can be expressed as\\ $C(b_0, b_m, ... , b_2, b_1, \overbrace{1, 1, 1, 1, 1}^{\text{5}}, 1-b_1, -b_2, ... , -b_m)$, \\
    where $b_i$ is an integer for $0 \leq i \leq m$. 
\end{enumerate}
\end{restatable}

\begin{restatable}{theorem}{deltaone}
(\cite{Naka3}) \label{thm:deltaone}
Let $K=C(b_0, b_m, ... , b_2, b_1, 1, 1, 1, 1, 1, 1-b_1, -b_2, ... , -b_m)$ be a nontrivial two-bridge knot, where $b_i$ is an integer for $0 \leq i \leq m$. 
Then, we have $u^{\Delta}(K)= 1$.
\end{restatable}

\subsection{Pretzel Knots} 
\label{subsec22}

\leavevmode

A pretzel link $P(p_1, p_2, ... , p_n)$ is a link obtained by connecting $n$ twisted bands, where each band has $p_i$ half-twists.
The sign of $p_i$ determines the direction of the $i$th twist region
(see Figure \ref{fig:pretzel}).
A pretzel link is a knot if and only if precisely one $p_i$ is even, or all $p_i$ are odd and $n$ is odd.
Previous work \cite{Naka4} computed the $\Delta$-unknotting numbers of positive pretzel knots.
\begin{figure}[htbp]
    \centering
    \includegraphics[width=0.8\linewidth]{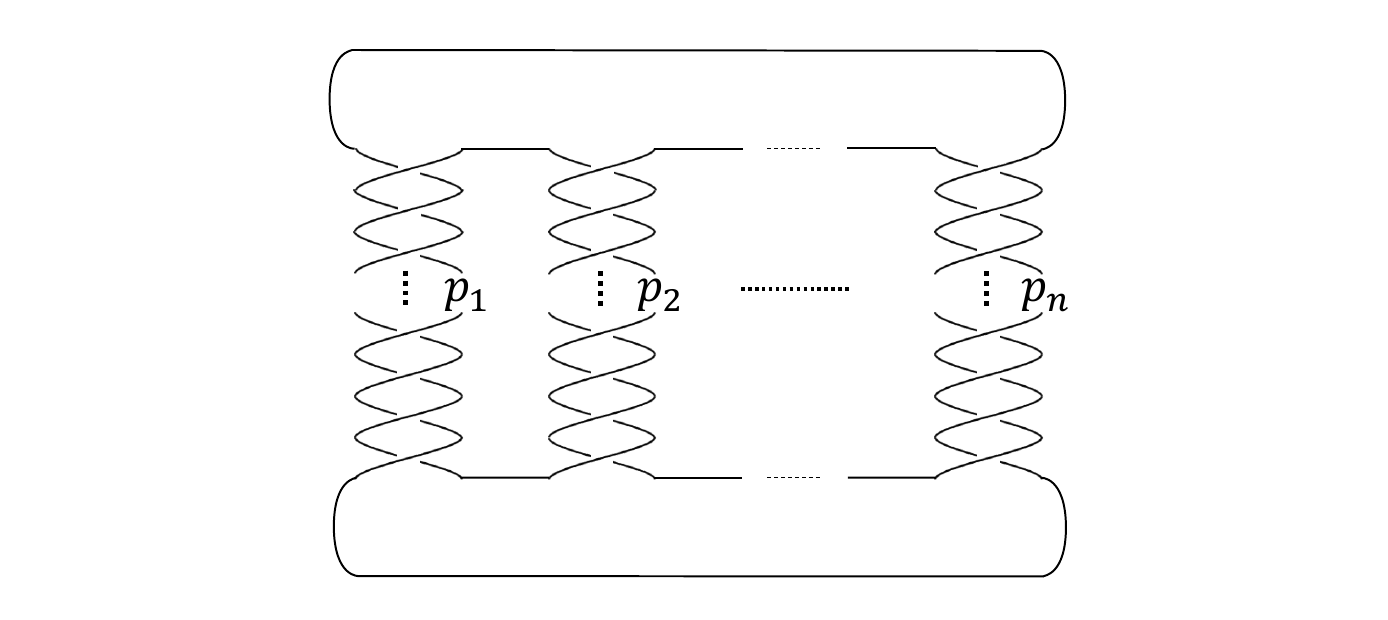}
    \caption{A pretzel knot $P(p_1, p_2, ... , p_n)$.}
    \label{fig:pretzel}
\end{figure}

\subsection{Montesinos Knots and Rational Tangles} 
\label{subsec23}

\leavevmode

A Montesinos knot $K = M\big(b ; (\alpha_1, \beta_1), (\alpha_2, \beta_2), \ldots, (\alpha_r, \beta_r)\big)$ with $r$ branches is a knot as illustrated in Figure \ref{fig:montesinos}.
Here $r$, $b$, $\alpha_k$, and $\beta_k$ are integers 
such that $r \geq 0$,  $\alpha_k \geq 2$, and $\gcd(\alpha_k, \beta_k)=1$.
The quantity $|b|$ denotes the number of half-twists corresponding to the integer parameter $b$, where a right-handed half-twist is considered positive.
Let $T(\beta_k / \alpha_k)$ denote the rational tangle of slope $\beta_k / \alpha_k$ (see Figure  \ref{fig:tangle}).
\begin{figure}[htbp]
    \centering
    \includegraphics[width=0.8\linewidth]{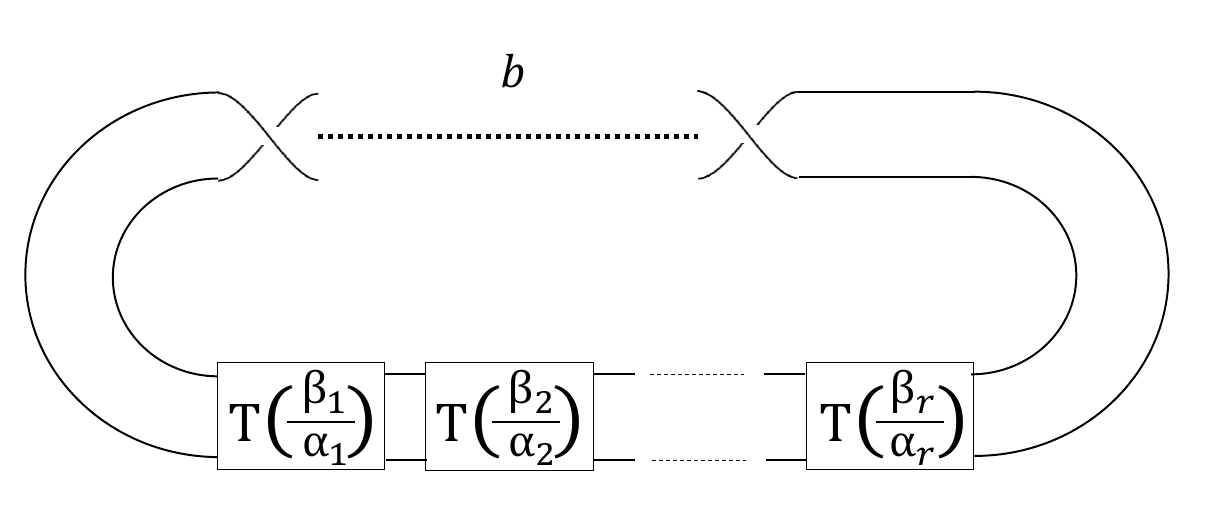}
    \caption{A Montesinos knot $M\big(b ; (\alpha_1, \beta_1), (\alpha_2, \beta_2), \ldots, (\alpha_r, \beta_r)\big)$.}
    \label{fig:montesinos}
\end{figure}
\begin{figure}[htbp]
    \centering
    \includegraphics[width=0.8\linewidth]{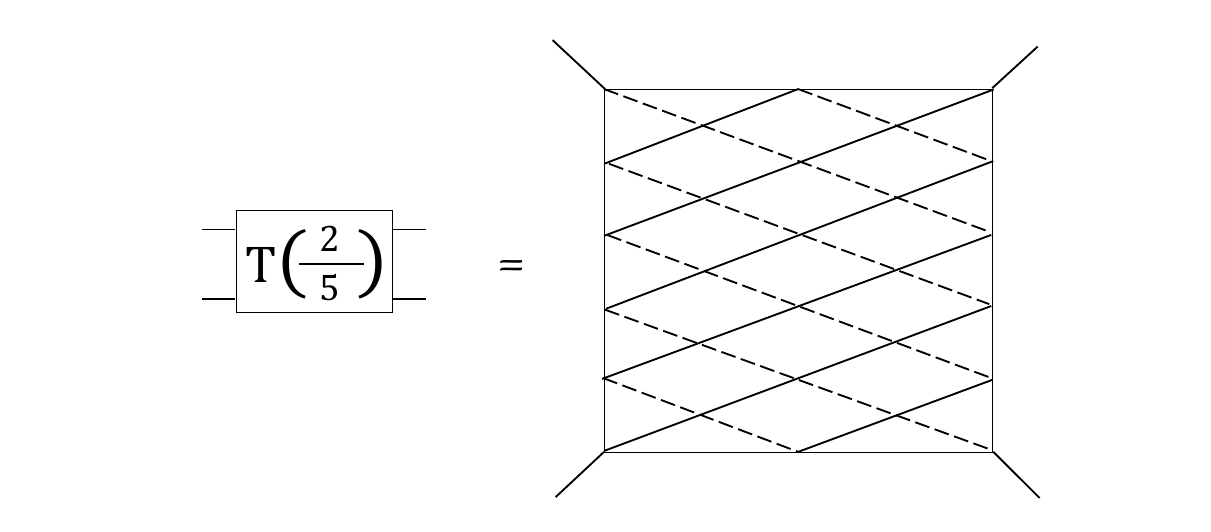}
    \caption{The tangle $T(2/5)$.}
    \label{fig:tangle}
\end{figure}
For the tangle $T(\beta_k / \alpha_k)$,
let $L \big( T(\beta_k / \alpha_k) \big)$ denote the link obtained by the standard left/right closure
(see Figure \ref{fig:closure}).
\begin{figure}[htbp]
    \centering
    \includegraphics[width=0.3\linewidth]{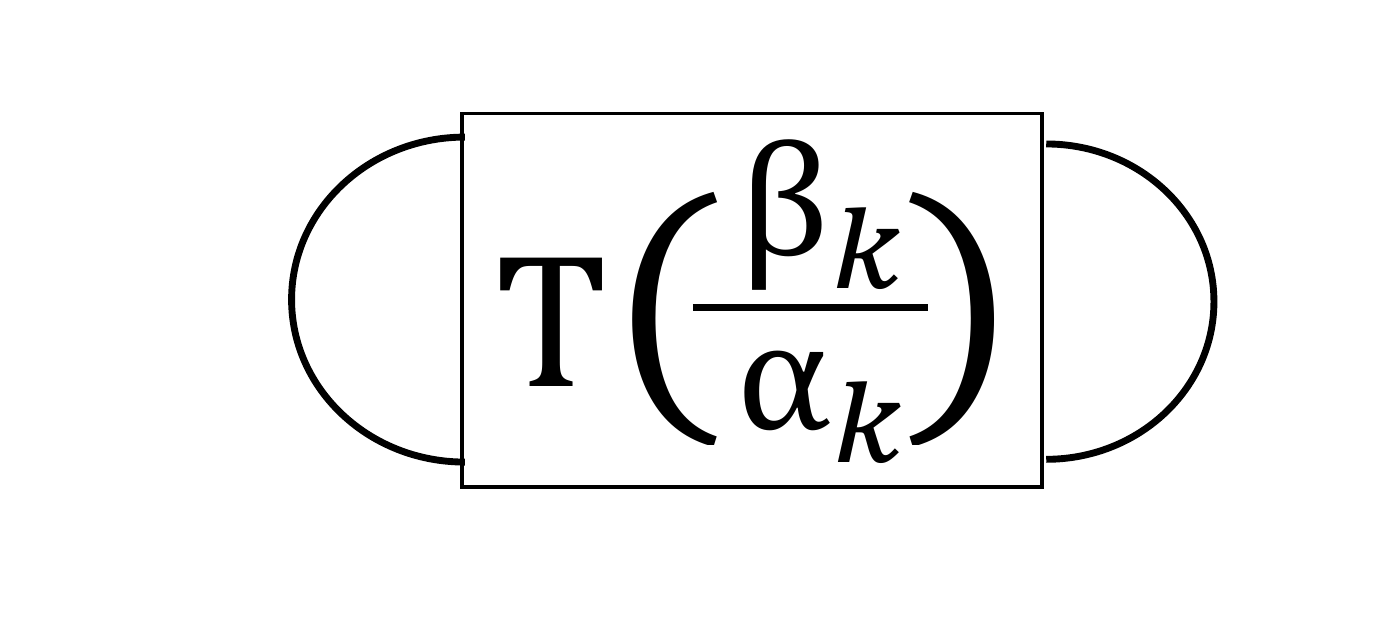}
    \caption{A link
    $L \big( T(\beta_k / \alpha_k) \big)$
    obtained by the standard left/right closure
    of the tangle $T(\beta_k / \alpha_k)$.}
    \label{fig:closure}
\end{figure}
A continued fraction expansion of a rational number $\beta / \alpha$, where we assume $-\alpha < \beta < \alpha$, is a finite sequence $s_1, s_2$, ... , $s_n$ such that
\begin{align*}
\frac{\beta}{\alpha}
& = \cfrac{1}{\,s_1 - \cfrac{1}{\,s_2 - \cfrac{1}{\,\ddots - \cfrac{1}{s_n}}}} \\
& = \cfrac{1}{\,s_1 + \cfrac{1}{\,- s_2 + \cfrac{1}{\,\ddots + \cfrac{1}{(-1)^{n-1} s_n}}}},
\end{align*}
and $s_i \neq 0$ for $1 \leq i \leq n.$
The tangle $T(\beta / \alpha)$ is then representable as in Figure \ref{fig:tangletwo} (see \cite{HaM}).
\begin{figure}[htbp]
    \centering
    \includegraphics[width=1\linewidth]{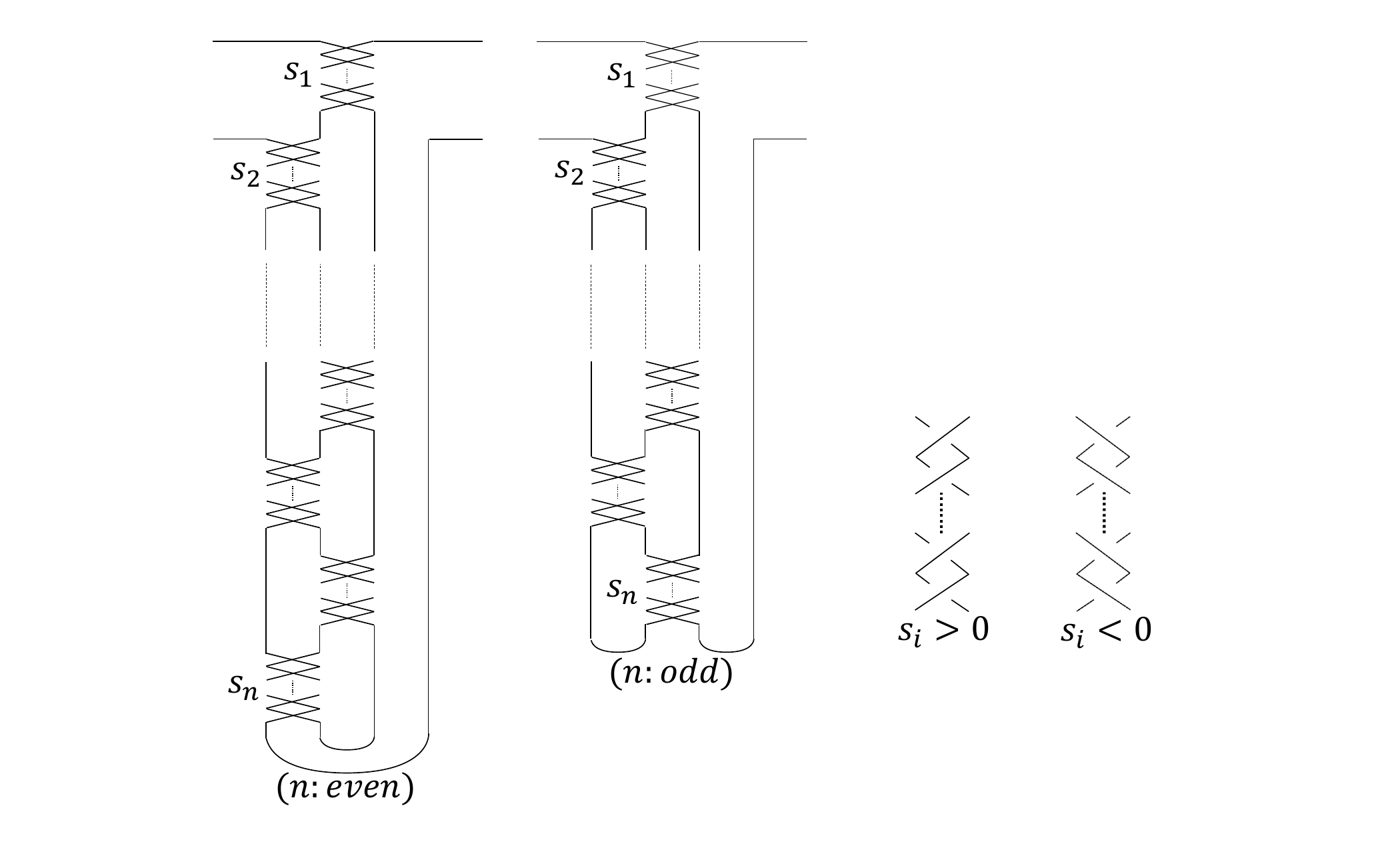}
    \caption{Another representation of a rational tangle for illustration.}
    \label{fig:tangletwo}
\end{figure}
In this paper, we consider the following Montesinos knot
$K = M(0; (\alpha_1, \beta_1), \dots, (\alpha_r, \beta_r))$,
where for each $1 \le k \le r$, the rational number $\beta_k/\alpha_k$
admits a finite continued fraction expansion
\[
\frac{\beta_k}{\alpha_k}
=
\cfrac{1}{c^{(k)}_1
 + \cfrac{1}{c^{(k)}_2
 + \cfrac{1}{\ddots
 + \cfrac{1}{c^{(k)}_{n(k)}}}}},
\]
where $n(k)$ is a positive integer, and $c^{(k)}_i \neq 0$ for $1 \leq i \leq n(k)$.

The tangle $T(\beta_k / \alpha_k)$ is represented as in Figure \ref{fig:tangletwoc}, 
a presentation adopted to facilitate comparison with two-bridge knots.
\begin{figure}[htbp]
    \centering
    \includegraphics[width=1\linewidth]{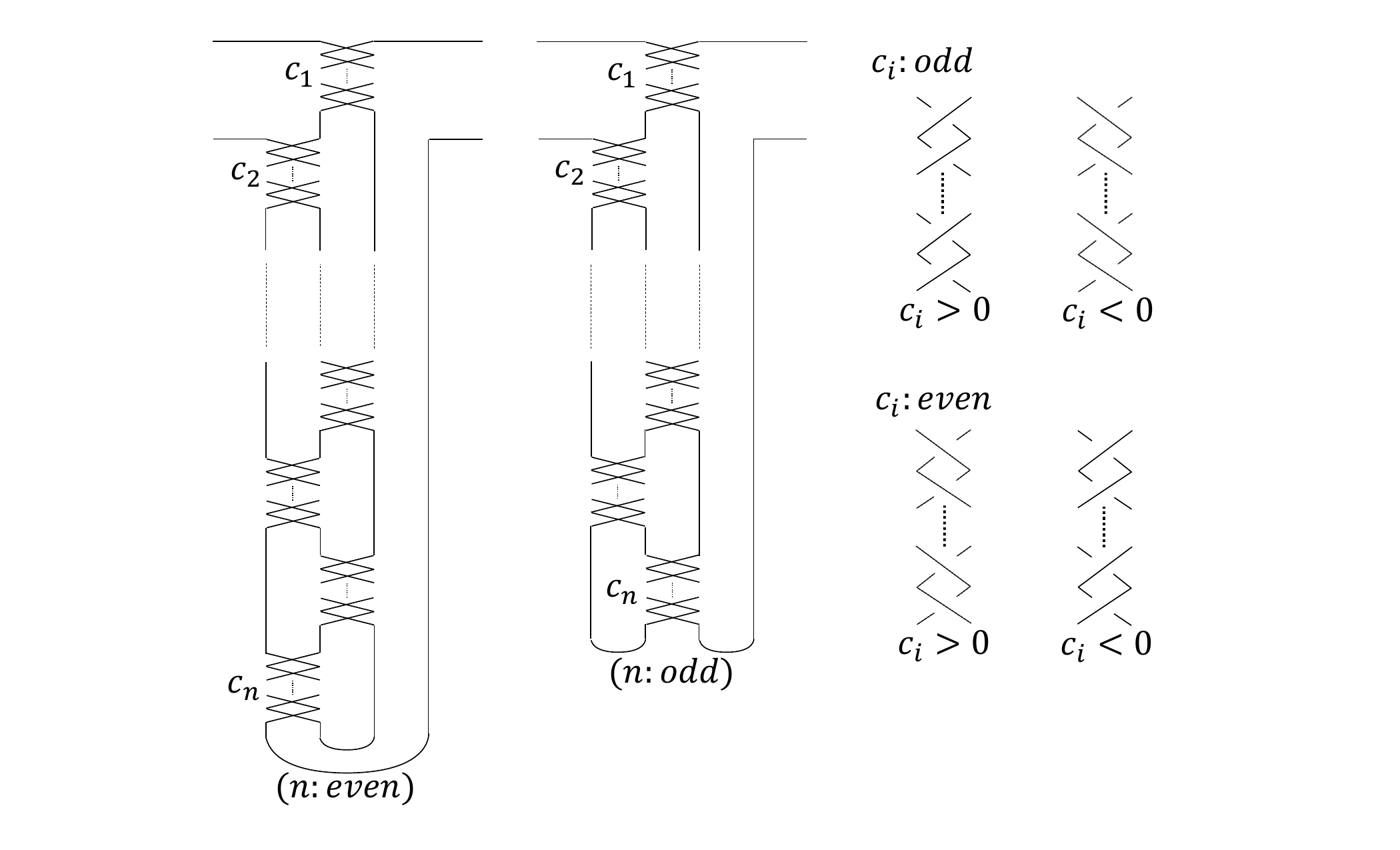}
    \caption{Another representation of a rational tangle used in this paper.}
    \label{fig:tangletwoc}
\end{figure}

\vspace{1em}
The upper left and lower left corners of $T(\beta / \alpha)$ are connected by a strand if and only if $\alpha$ is even.
Hence we see that if $K = M\big(0 ; (\alpha_1, \beta_1), (\alpha_2, \beta_2), \ldots, (\alpha_r, \beta_r)\big)$
is to be a knot rather than a link, at most one of $\alpha_1, \alpha_2$, ... , $\alpha_r$ can be even.
Since a cyclic permutation indices does not change the knot type, we hereafter assume that $\alpha_2$, ..., $\alpha_r$ are odd.
With this convention, we say that $K$ is of odd type if $\alpha_1$ is odd, and of even type if $\alpha_1$ is even.

Since $a_2(K)=a_2(K^*)$ and $u^{\Delta}(K)=u^{\Delta}(K^*)$, where $K^*$ denotes the mirror image of a knot $K$, we do not distinguish a knot from its mirror image in this paper.

\section{Montesinos Knots with $\Delta$-Unknotting Number One}\label{sec3}

In this section, we discuss Montesinos knots with $\Delta$-unknotting number one.

We consider a tangle of Type I.

A tangle $T(\beta / \alpha )$ is of Type I if $(c_1, c_2, ..., c_n) = (b_0, b_m, ... , b_2, b_1, 1, 1, 1, 1, 1, 1-b_1, -b_2, ..., -b_m)$, where $n=2m+6$
(see Figure \ref{fig:tangletwoc}), and  $b_i$ is an integer for $0 \leq i \leq m$.

$T(\beta/\alpha)$ can be deformed into the horizontal tangle
(see Figure \ref{fig:horizontal})
by a single $\Delta$-move.

By Theorem \ref{thm:deltaone}, 
$u^{\Delta}(L \big( T(\beta / \alpha) \big)) = 1$.
\begin{figure}[htbp]
    \centering
    \includegraphics[width=0.3\linewidth]{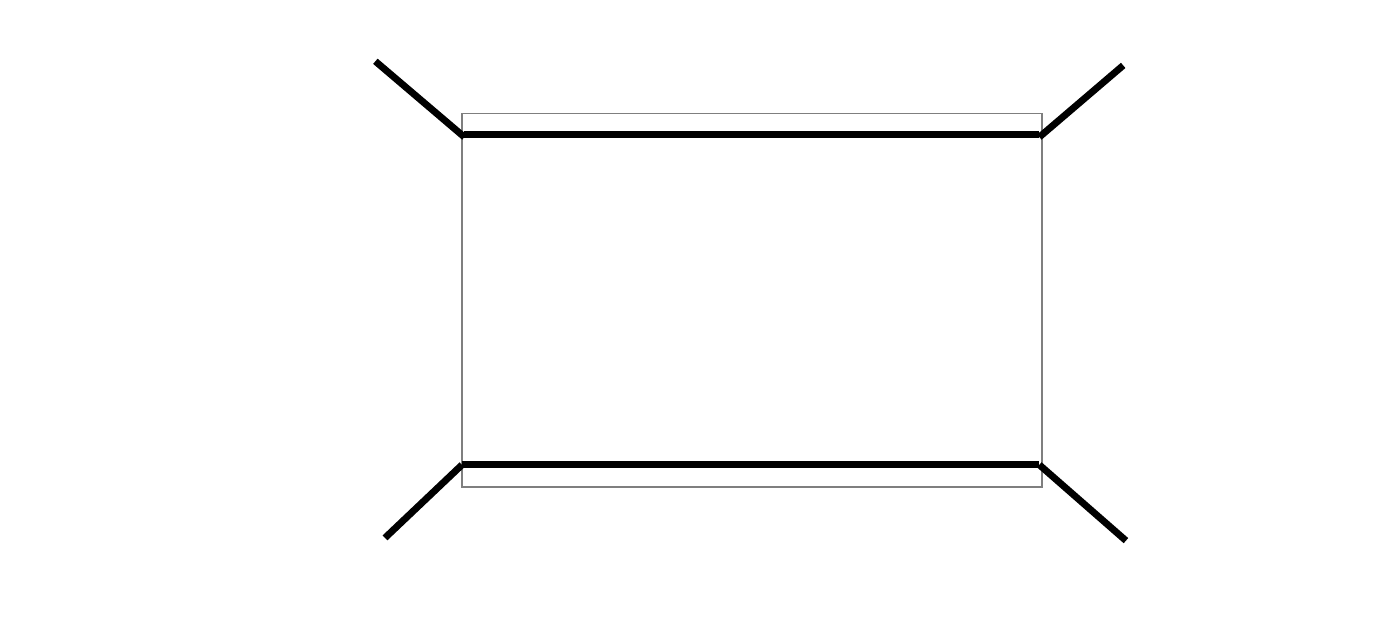}
    \caption{The horizontal tangle.}
    \label{fig:horizontal}
\end{figure}

\vspace{1em}
The following conjecture and theorem for Montesinos knots are motivated by Conjecture \ref{thm:conj1} and Theorem \ref{thm:deltaone}.

\conjmone*

\mone*

\begin{proof}
$T(\beta_3 / \alpha_3)$ can be deformed into the horizontal trivial tangle by a single $\Delta$-move.

Hence, applying a $\Delta$-move to the corresponding tangle component $T(\beta_3 / \alpha_3)$ in
$M(0;(\alpha_1, \beta_1), (\alpha_2, \beta_2), (\alpha_3, \beta_3))$,
we obtain the Montesinos knots
$M(0;(\alpha_1, \beta_1), (\alpha_2, \beta_2))$.
It is known that the two-bridge knot denoted by $M(0;(\alpha_1, \beta_1), (\alpha_2, \beta_2))$ is the trivial knot if and only if $|\alpha_1 \beta_2 + \beta_1 \alpha_2| = 1$ (see \cite{Tor}).
Since this condition is satisfied, the resulting knot is the trivial knot.

The proof is complete.
\end{proof}


Theorem \ref{thm:moneone} follows from Theorem \ref{thm:mone}.

\moneone*

\begin{proof}
Both tangles $T(\beta_2 / \alpha_2 )$ and $T(\beta_3 / \alpha_3)$ can be deformed into the horizontal trivial tangle by a single $\Delta$-move.

Hence, applying a $\Delta$-move to the corresponding tangle component $T(\beta_2 / \alpha_2)$ (resp. $T(\beta_3 / \alpha_3)$) in
$M(0;(\alpha_1, \beta_1), (\alpha_2, \beta_2), (\alpha_3, \beta_3))$,
we obtain the Montesinos knots
$M(0;(\alpha_1, \beta_1), (\alpha_3, \beta_3))$
and
$M(0;(\alpha_1, \beta_1), (\alpha_2, \beta_2))$, respectively.
These knots are trivial knots.

The proof is complete.
\end{proof}

The following examples satisfy the conditions of Theorem \ref{thm:moneone}.

Examples \ref{ex31}, \ref{ex32}, and \ref{ex33} illustrate the case
$T(\beta_2 / \alpha_2) \cong
T(\beta_3 / \alpha_3)$
in Theorem \ref{thm:moneone}. 

\begin{example}\label{ex31}
Let $K = M(0;(\pm (3 b_0 +5), \mp 3), (\pm (8 b_0 + 13), \pm 8), (\pm (8 b_0 + 13), \pm 8))$, where $b_0$ is an integer
(see Tables \ref{table1} and \ref{table2}).
Then $u^{\Delta}(K) = 1$.

In particular, performing a $\Delta$-move in either of the two tangles $T(\pm (8 b_0 + 13), \pm 8) $ yields the trivial knot.

Furthermore, $u(K)=1$,
since $(8 b_0 + 13, 8)$ belongs to the family $(2mn \pm 1,2n^2)$
by taking $n=2$ and $m=2b_0 + 3$
(see \cite{Tor}).

For example, if $b_0 = -1$, 
then $K = M(0;(2,-3),(5,8),(5,8))\cong 10_{60}$ $(\cong M(0;(2,1),(5,3),(5,3)) )$
(see Figure \ref{fig:1060}).
\end{example}

\begin{table}[htbp]
\centering
\caption{$M(0;(3 b_0 +5, -3), (8 b_0 + 13, 8), (8 b_0 + 13, 8))$}
\label{table1}
\begin{tabular}{c|c||c}
\hline
$L \big( T(\beta / \alpha) \big)$ & $\beta / \alpha$ & Montesinos knot \\
\hline
$(4_1 \cong C(2,2) \cong) \quad C(-1,0,1,1,1,1,1,1)$ &
$8/5$ & $M(0;(2,-3),(5,8),(5,8))$ \\
\hline
$(6_3 \cong C(2,1,1,2) \cong) \quad C(0,0,1,1,1,1,1,1)$ &
$8/13$ & $M(0;(5,-3),(13,8),(13,8))$ \\
\hline
$(7_7 \cong C(2,1,1,1,2) \cong) \quad C(1,0,1,1,1,1,1,1)$ &
8/21 &
$M(0;(8,-3),(21,8),(21,8))$ \\
\hline
$(8_{13} \cong C(3,1,1,1,2) \cong) \quad C(2,0,1,1,1,1,1,1)$ &
8/29 &
$M(0;(11,-3),(29,8),(29,8))$ \\
\hline
$(9_{14} \cong C(4,1,1,1,2) \cong) \quad C(3,0,1,1,1,1,1,1)$ &
8/37 &
$M(0;(14,-3),(37,8),(37,8))$ \\
\hline
$(10_{10} \cong C(5,1,1,1,2) \cong) \quad C(4,0,1,1,1,1,1,1)$ &
8/45 &
$M(0;(17,-3),(45,8),(45,8))$ \\
\hline
\end{tabular}
\end{table}

\begin{table}[htbp]
\centering
\caption{$M(0;(-3 b_0 -5, 3), (-8 b_0 - 13, -8), (-8 b_0 - 13, -8))$}
\label{table2}
\begin{tabular}{c|c||c}
\hline
$L \big( T(\beta / \alpha) \big)$ & $\beta / \alpha$ & Montesinos knot \\
\hline
$(3_1 \cong C(3) \cong) \quad C(-2,0,1,1,1,1,1,1)$ &
-8/3 &
$M(0;(1,3),(3,-8),(3,-8))$ \\
\hline
$(6_2 \cong C(3,1,2) \cong) \quad C(-3,0,1,1,1,1,1,1)$ &
-8/11 &
$M(0;(4,3),(11,-8),(11,-8))$ \\
\hline
$(7_6 \cong C(2,2,1,2) \cong) \quad C(-4,0,1,1,1,1,1,1)$ &
-8/19 &
$M(0;(7,3),(19,-8),(19,-8))$ \\
\hline
$(8_{11} \cong C(3,2,1,2) \cong) \quad C(-5,0,1,1,1,1,1,1)$ &
-8/27 &
$M(0;(10,3),(27,-8),(27,-8))$ \\
\hline
$(9_{12} \cong C(4,2,1,2) \cong) \quad C(-6,0,1,1,1,1,1,1)$ &
-8/35 &
$M(0;(13,3),(35,-8),(35,-8))$ \\
\hline
$(10_{7} \cong C(5,2,1,2) \cong) \quad C(-7,0,1,1,1,1,1,1)$ &
-8/43 &
$M(0;(16,3),(43,-8),(43,-8))$ \\
\hline
\end{tabular}
\end{table}

\begin{example}\label{ex32}
Let $K = M(0;(\pm (11 b_0 +3), \mp 11), (\pm (40 b_0 + 11), \pm 40), (\pm (40 b_0 + 11), \pm 40))$, where $b_0$ is an integer
(see Table \ref{table32}).
Then $u^{\Delta}(K) = 1$.

In particular, performing a $\Delta$-move in either of the two tangles $T(\pm (40 b_0 + 11), \pm 40) $ yields the trivial knot.
\end{example}

\begin{table}[htbp]
\centering
\caption{$M(0;(\pm (11 b_0 +3), \mp 11), (\pm (40 b_0 + 11), \pm 40), (\pm (40 b_0 + 11), \pm 40))$}
\label{table32}
\begin{tabular}{c|c||c}
\hline
$L \big( T(\beta / \alpha) \big)$ & $\beta / \alpha$ & Montesinos knot \\
\hline
$(10_{19} \cong C(4,1,1,1,3) \cong) \quad C(1,3,1,1,1,1,1,-2)$ &
$40/51$ & $M(0;(14,-11),(51,40),(51,40))$ \\
\hline
$(10_{32} \cong C(3,1,1,1,2,2) \cong) \quad C(-2,3,1,1,1,1,1,-2)$ &
$-40/69$ & $M(0;(19,11),(69,-40),(69,-40))$ \\
\hline
\end{tabular}
\end{table}


\begin{example}\label{ex33}
Let
$K = M(0;(5m \pm 2, -8m \mp 3), (5, 8), (5, 8))$
or
$K = M(0;(3m \pm 1, 8m \pm 3), (3, -8), (3, -8))$, where $m$ is an integer.
Then $u^{\Delta}(K) = 1$.

In particular, 
in the first case, performing a $\Delta$-move in either of the two tangles $T(5,8)$,
and in the second case, performing a $\Delta$-move in either of the two tangles $T(3,-8)$,
yields the trivial knot.

For example, $M(0;(0+2,0-3),(5,8),(5,8))
\cong M(0;(2,-3),(5,8),(5,8)) \cong 10_{60}$, 
and $M(0;(3-1,8-3),(3,-8),(3,-8))^*
\cong M(0;(2,5),(3,-8),(3,-8))^* \cong 9_{28}$.

\end{example}

\vspace{1em}
Examples \ref{ex34} and \ref{ex35} illustrate the case
$T(\beta_2 / \alpha_2) \not\cong
T(\beta_3 / \alpha_3)$
in Theorem \ref{thm:moneone}. 

\begin{example}\label{ex34}
Let $K=M(0;(14,-3),(37,8),(51,11))$(see Table \ref{table34}).
Then $u^{\Delta}(K) = 1$.
In particular, performing a $\Delta$-move in either of the two tangles $T(8, 37)$ and $T(11, 51)$ yields the trivial knot.
\end{example}

\begin{table}[htbp]
\centering
\caption{$M(0;(14, -3), (14m + 23, 3m + 5), (14n + 23, 3n + 5))$}
\label{table34}
\begin{tabular}{c|c||c}
\hline
$L \big( T(\beta / \alpha) \big)$ & $\beta / \alpha$ & Montesinos knot \\
\hline
$9_{14} \cong C(4,1,1,1,2)$ &
8/37 &
$M(0;(14,-3),(37,8),(37,8))$ \\
$10_{19} \cong C(4,1,1,1,3)$ &
$11/51$ & $M(0;(14,-3),(51,11),(51,11))$ \\
 &
 &
$M(0;(14,-3),(37,8),(51,11))$ \\
\hline
\end{tabular}
\end{table}

\begin{example}\label{ex35}
Let $K=M(0;(98,-27),(29,8),(69,19))$(see Table \ref{table35}).
Then $u^{\Delta}(K) = 1$.
In particular, performing a $\Delta$-move in either of the two tangles $T(8, 29)$ and $T(19, 69)$ yields the trivial knot.
\end{example}

\begin{table}[htbp]
\centering
\caption{$M(0;(98, -27), (40m - 11, 11m - 3), (40n -11, 11n - 3))$}
\label{table35}
\begin{tabular}{c|c||c}
\hline
$L \big( T(\beta / \alpha) \big)$ & $\beta / \alpha$ & Montesinos knot \\
\hline
$8_{13} \cong C(3,1,1,1,2)$ &
8/29 &
$M(0;(98,-27),(29,8),(29,8))$ \\
$10_{32} \cong C(3,1,1,1,2,2)$ &
$19/69$ & $M(0;(98,-27),(69,19),(69,19))$ \\
 &
 &
$M(0;(98,-27),(29,8),(69,19))$ \\
\hline
\end{tabular}
\end{table}

\vspace{1em}
\begin{remark}\label{remark:delta}
The following observations hold. Statements {\rm (1)} and {\rm (2)} are equivalent, and so are statements {\rm (3)} and {\rm (4)} (see Figures \ref{fig:tangle53} and \ref{fig:tangle32}.)

\begin{enumerate}
\item $T(3/5)$ can be deformed into $T(-1/1)$ by a single $\Delta$-move.

\item $T(8/5)$ can be deformed into the horizontal tangle by a single $\Delta$-move.

\item $T(2/3)$ can be deformed into $T(-2/1)$ by a single $\Delta$-move.

\item $T(8/3)$ can be deformed into the horizontal tangle by a single $\Delta$-move.
\end{enumerate}

\end{remark}
\begin{figure}[htbp]
    \centering
    \includegraphics[width=0.6\linewidth]{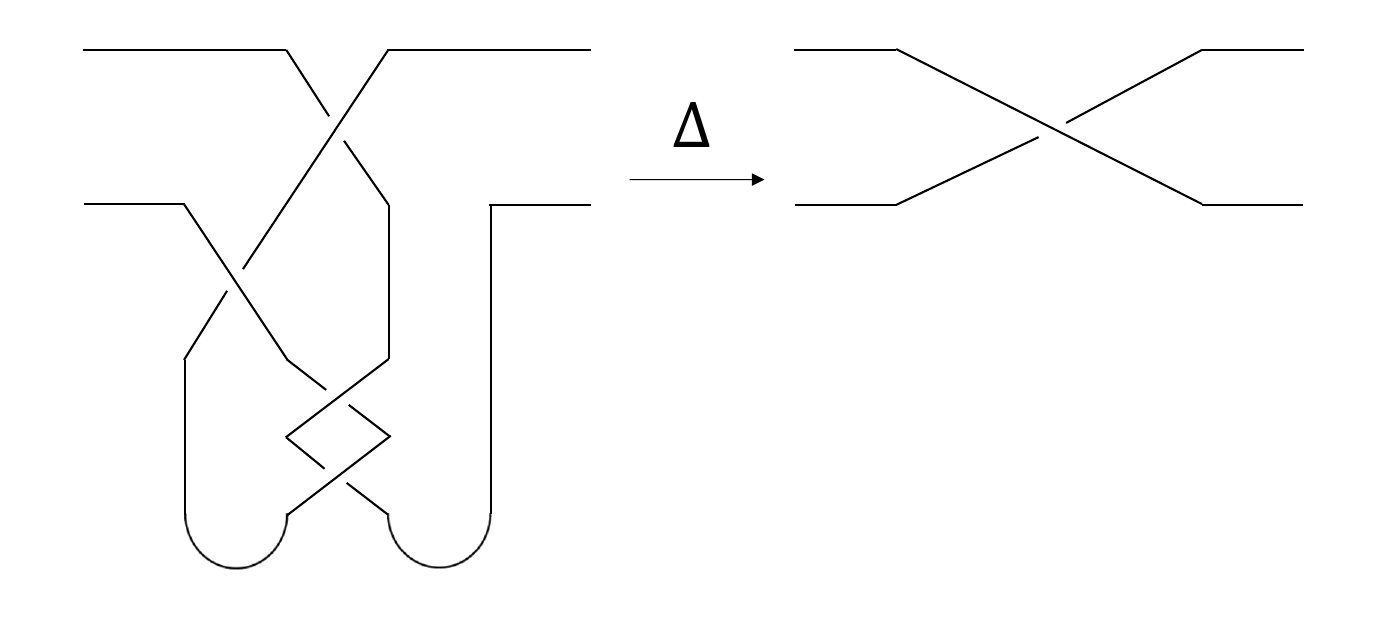}
    \caption{Deformation of $T(3/5)$.}
    \label{fig:tangle53}
\end{figure}
\begin{figure}[htbp]
    \centering
    \includegraphics[width=0.6\linewidth]{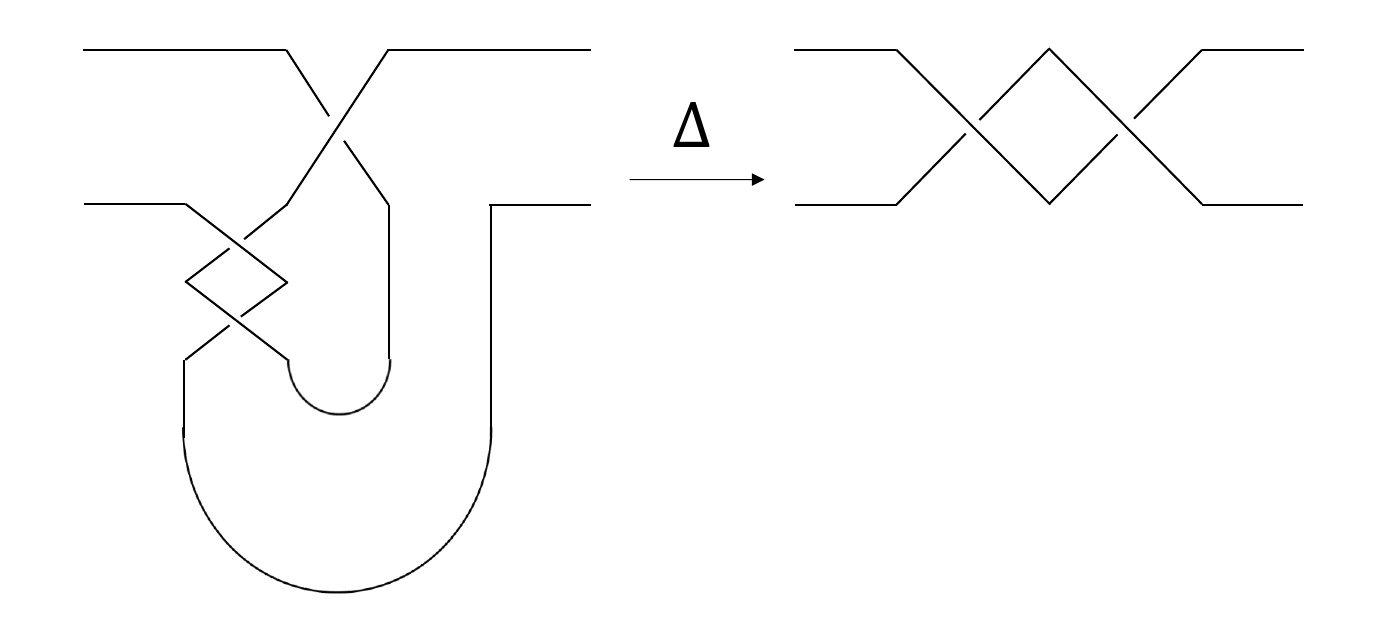}
    \caption{Deformation of $T(2/3)$.}
    \label{fig:tangle32}
\end{figure}

\begin{example}\label{exmone}
By Remark \ref{remark:delta}, the deformations of Montesinos knots with $u^{\Delta}(K)=1$ and at most 10 crossings are listed in Table \ref{tabledeform}.

In particular, $9_{28}$ and $10_{60}$ (see Figure \ref{fig:1060}) satisfy the conditions of Theorem \ref{thm:moneone}.

However, we have not been able to obtain deformation descriptions for 
$10_{141}$ and $10_{147}$
in the form shown in Table \ref{tabledeform}, whereas we have obtained such descriptions for the other knots.
\end{example}

\begin{table}[htbp]
\centering
\caption{Deformations of Montesinos knots with $u^{\Delta}(K)=1$}
\label{tabledeform}
\begin{tabular}{c|c|c}
\text{Montesinos knot} & $\Delta$ & \text{trivial knot} \\
\hline
$9_{22} \cong M(0;(2,1),(3,1),\underline{(5,3)})$
& $\to$ &
$M(1;(2,1),(3,1)) \cong M(0;(2,1),(3,-2))$ \\
$9_{24} \cong M(0;(3,1),\underline{(3,2)},(2,3))$
& $\to$ &
$M(2;(3,1),(2,3)) \cong M(0;(3,1),(2,-1))$ \\
$9_{28} \cong M(0;\underline{(3,2)},\underline{(3,2)},(2,3))$
& $\to$ &
$M(2;(3,2),(2,3)) \cong M(0;(3,2),(2,-1))$ \\
$9_{30} \cong M(0;(2,1),(3,2),\underline{(5,3)})$
& $\to$ &
$M(1;(2,1),(3,2)) \cong M(0;(2,1),(3,-1))$ \\
$10_{59} \cong M(0;(2,1),(5,2),\underline{(5,3)})$
& $\to$ &
$M(1;(2,1),(5,2)) \cong M(0;(2,1),(5,-3))$ \\
$10_{60} \cong M(0;(2,1),\underline{(5,3)},\underline{(5,3)})$
& $\to$ &
$M(1;(2,1),(5,3)) \cong M(0;(2,1),(5,-2))$ \\
$10_{71} \cong M(0;(5,2),\underline{(3,2)},(2,3))$
& $\to$ &
$M(2;(5,2),(2,3)) \cong M(0;(5,-8),(2,3))$ \\
$10_{73} \cong M(0;(5,3),\underline{(3,2)},(2,3))$
& $\to$ &
$M(2;(5,3),(2,3)) \cong M(0;(5,-7),(2,3))$ \\
\end{tabular}
\end{table}

\begin{remark}\label{remark:10141}
We consider $10_{141}$.
We can deform $10_{141} \cong M(0;(3,-1),(3,2),(4,1))$ into $M(0;(5,-1),(4,1)) \cong O$ by a single $\Delta$-move.
This deformation is different from the pattern in Example \ref{exmone}. Therefore, $10_{141}$ may provide a counterexample to
Conjecture \ref{thm:conjmone}.

On the other hand, the Conway polynomial of $10_{141}$ is $1-z^2-3z^4-z^6$, which coincides with that of $8_5$.
However, $u^{\Delta}(10_{141})=1$ and $u^{\Delta}(8_{5})=3$.
Thus, the Conway polynomial alone cannot determine the $\Delta$-unknotting number.

Similarly, we can deform $12n_{571} \cong M(0;(3,-1),(3,2),(6,1))$ into $M(0;(5,-1),(6,1)) \cong O$ by a single $\Delta$-move.
\end{remark}

\begin{table}[htbp]
\centering
\caption{Deformations of Montesinos knots with $u^{\Delta}(K)=1$}
\label{tabledeform}
\begin{tabular}{c|c|c}
\text{Montesinos knot} & $\Delta$ & \text{trivial knot} \\
\hline
$10_{141} \cong M(0;(3,-1),(3,2),(4,1))$
& $\to$ &
$M(0;(5,-1),(4,1))$ \\
$12n_{571} \cong M(0;(3,-1),(3,2),(6,1))$
& $\to$ &
$M(0;(5,-1),(6,1))$ \\
\end{tabular}
\end{table}

In connection with Theorem \ref{thm:moneone}, we pose the following problem.

\begin{restatable}{problem}{moneoneone} \label{thm:moneoneone}
Does there exist a nontrivial Montesinos knot $M(0;(\alpha_1, \beta_1), (\alpha_2, \beta_2), (\alpha_3, \beta_3))$ such that
$T(\beta_1 / \alpha_1)$, $T(\beta_2 / \alpha_2)$, and $T(\beta_3 / \alpha_3)$ are of Type I, 
$|\alpha_1 \beta_2 + \beta_1 \alpha_2| = |\alpha_1 \beta_3 + \beta_1 \alpha_3| =  |\alpha_2 \beta_3 + \beta_2 \alpha_3| = 1$, and
$\alpha_1, \alpha_2, \alpha_3, \beta_1, \beta_2, \beta_3$ are non zero integers?
\end{restatable}

\vspace{1em}
\section{$X$-moves and $\Delta$-moves for knots with $\Delta$-Unknotting Number One}\label{sec4}

It is known that
$\frac{1}{2} u(K) \leq u^{\Delta}(K)$.
Therefore, if $u^{\Delta}(K)=1$, then $u(K)=1$ or $2$.
In this section, we present examples illustrating the relationship between the locations of 
$\Delta$-moves and $X$-moves for knots with
$u^{\Delta}(K) = 1$ and $u(K)=1$ or $2$.

\subsection{$\Delta$-Unknotting Number One and Unknotting Number One} 
\label{subsec41}

\leavevmode

\vspace{1em}
We use the technique in \cite{NaNaU}.

\begin{claim}[\cite{NaNaU}]\label{claim:change}
We can exchange a crossing of a knot diagram by a finite number of $\Delta$-moves.
Moreover, the number of $\Delta$-moves needed is equal to the number of hurdles.
\end{claim}

The meaning of this claim is illustrated in Figure \ref{fig:tdeltan}. 
If there are $t$ hurdles, the crossing labeled with the asterisk * is deformed into the opposite crossing by $t$ times $\Delta$-moves.
In particular, if there is exactly one hurdle, the crossing labeled with the asterisk * is deformed into the opposite crossing by a single $\Delta$-move.
\begin{figure}[htbp]
    \centering
    \includegraphics[width=0.8\linewidth]{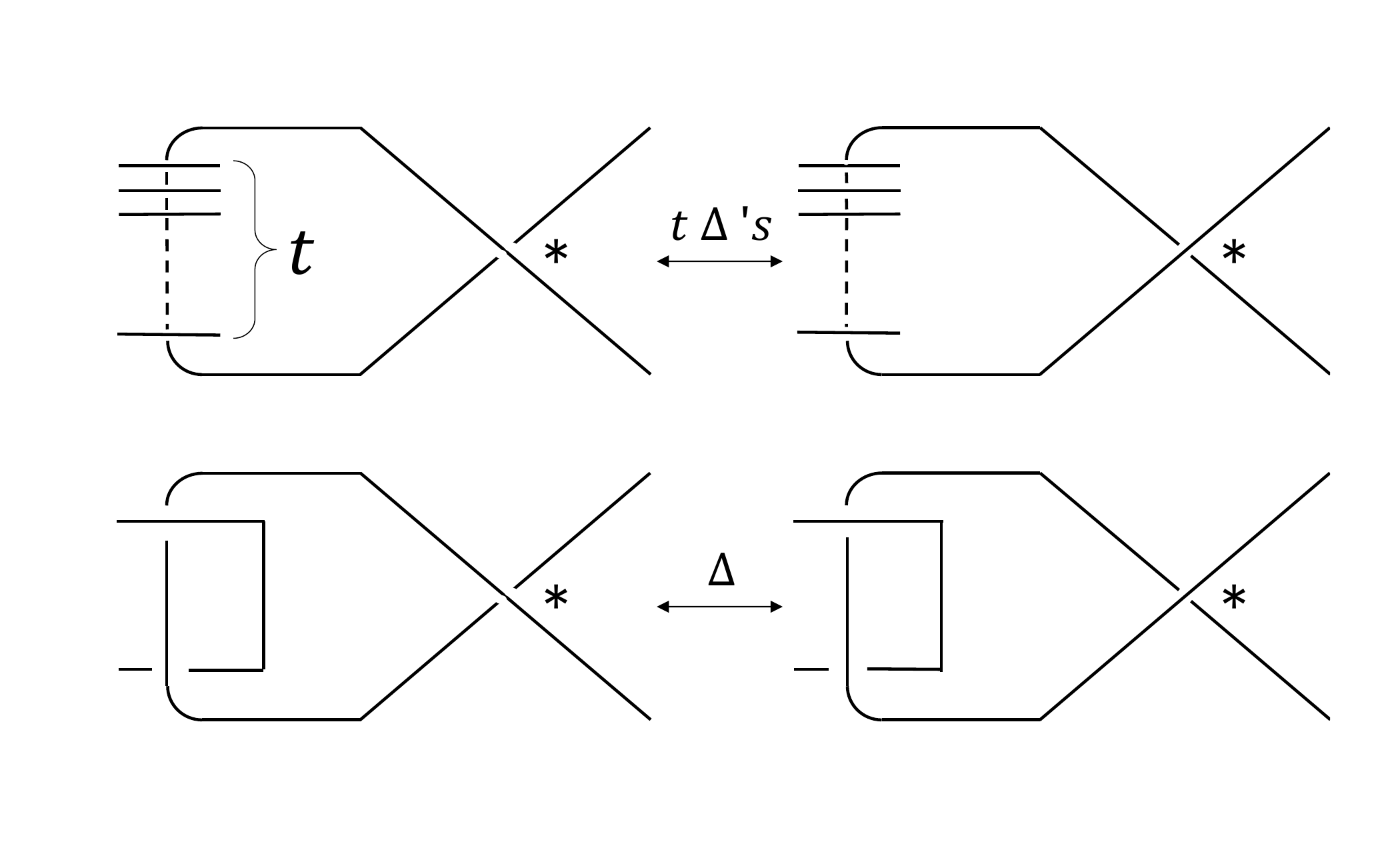}
    \caption{A crossing change realized by $t$ times $\Delta$-moves.}
    \label{fig:tdeltan}
\end{figure}

\begin{example}\label{remark:oneone}
Let $K$ be a two-bridge knot, a pretzel knot, or a Montesinos knot with $u^{\Delta}(K) = u(K) = 1$ and at most 9 crossings.
Then there exists a diagram of $K$ containing a crossing such that changing this crossing deforms $K$ into the trivial knot. 
Moreover, this crossing change can be realized by a single $\Delta$-move as illustrated in Figure \ref{fig:tdeltan}.

For example, this is the case for $9_{28}$ (see Figure \ref{fig:928o}).

Furthermore, the following families of knots also have this property: 
\begin{enumerate} 

\item 
$K = C(m,0,1,1,1,1,1,1)=S(8m+13,8)$, where $m$ is an integer(see Figure \ref{fig:m0111111}).
Here, $S(p,q)$ denotes Schubert's normal form.

\item 
$K = P(m,1,1 -3,1)$, where $m$ is a nonzero integer
(see Figure \ref{fig:m11n}).

\item $K = M(0;(\pm (3m +5), \mp 3), (\pm (8m + 13), \pm 8), (\pm (8m + 13), \pm 8))$, where $m$ is an integer (see Example \ref{ex31}).

\end{enumerate} 


\end{example}

\begin{figure}[htbp]
    \centering
    \includegraphics[width=1\linewidth]{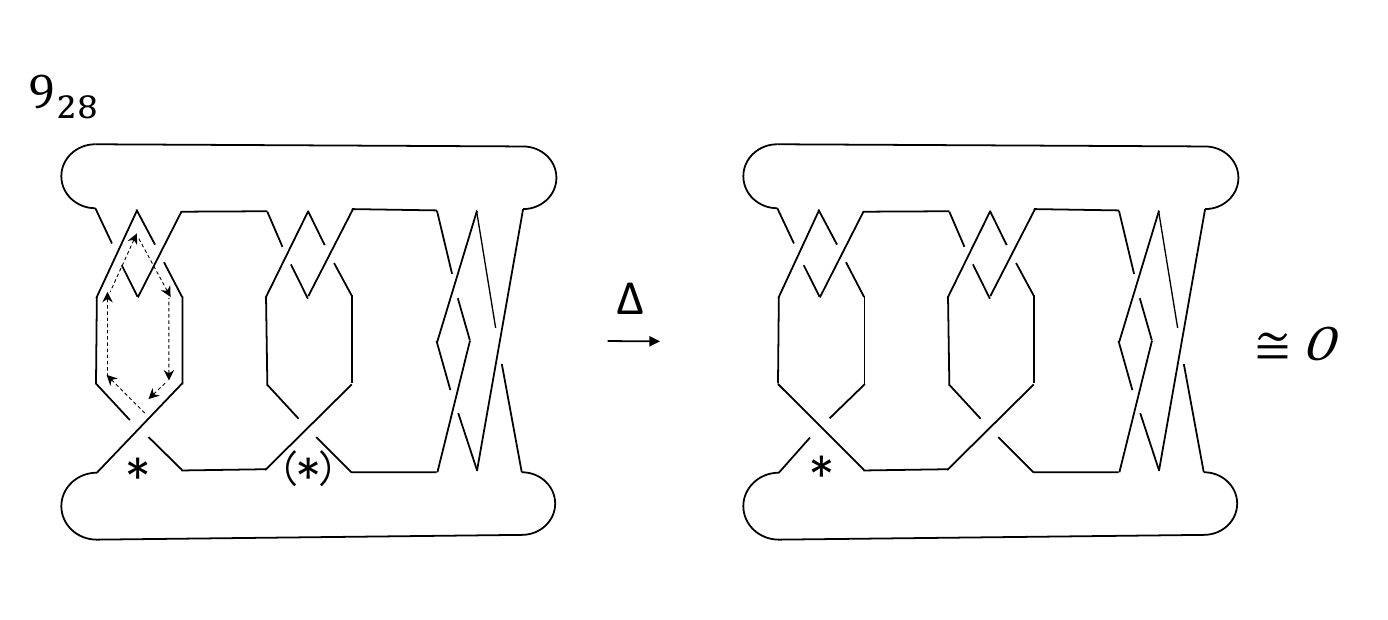}
    \caption{Deformation $9_{28}$ into $O$.}
    \label{fig:928o}
\end{figure}

\begin{figure}[htbp]
    \centering
    \includegraphics[width=0.8\linewidth]{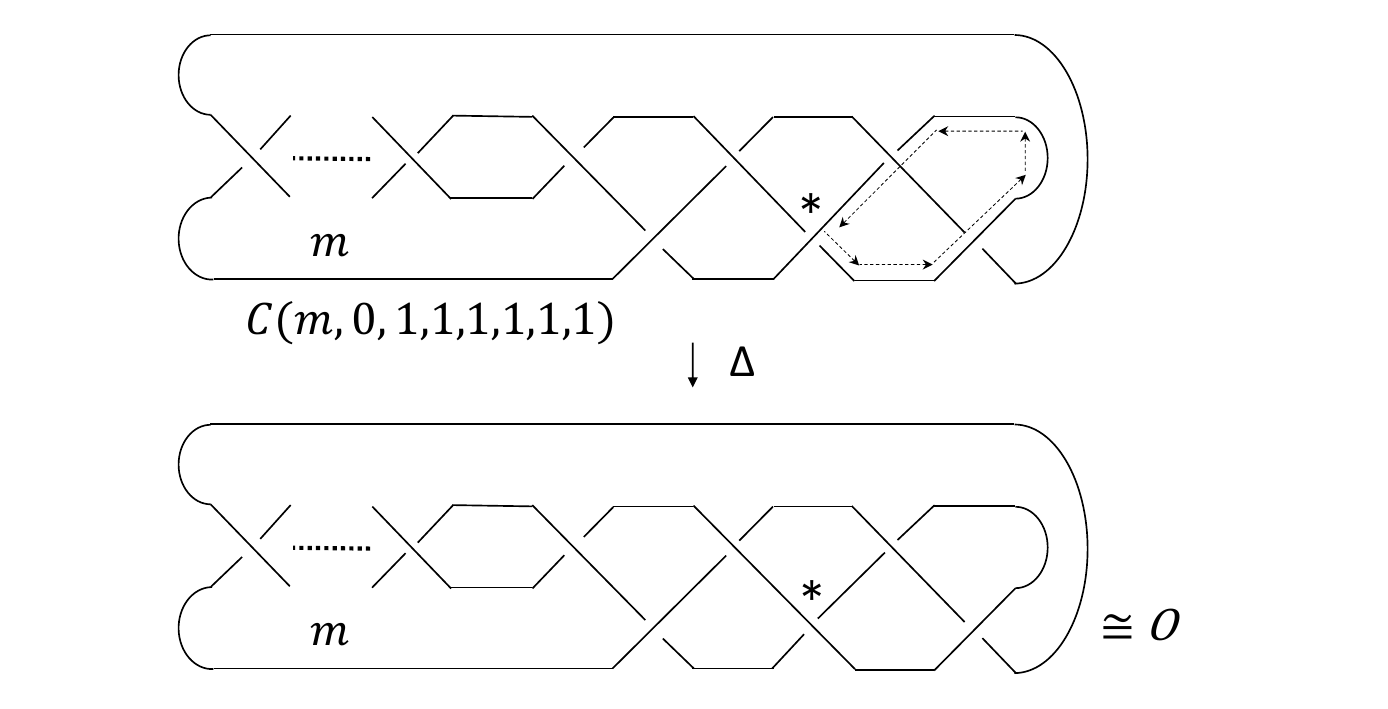}
    \caption{Deformation $C(m,0,1,1,1,1,1,1)$ into $O$.}
    \label{fig:m0111111}
\end{figure}

\begin{figure}[htbp]
    \centering
    \includegraphics[width=1\linewidth]{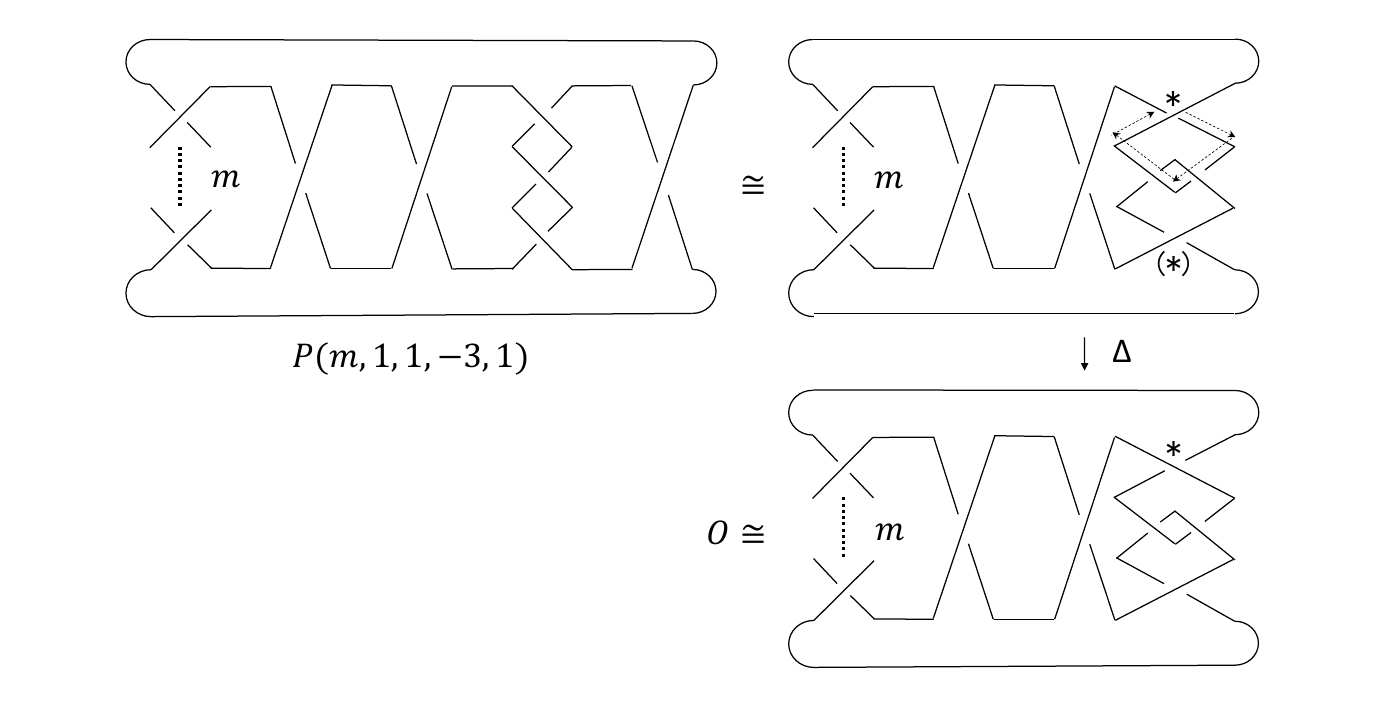}
    \caption{Deformation $P(m,1,1,-3,1)$ into $O$.}
    \label{fig:m11n}
\end{figure}

\begin{example}\label{remark:oneone2}
Let $K$ be a knot with $u^{\Delta}(K) = u(K) = 1$, and at most 10 crossings excluding the knots listed in Example \ref{remark:oneone} and the knots $10_{141}$ and $10_{147}$, and $10_{156}$.
Then there exists a diagram of $K$ containing a crossing $X_r$, $X_l$, $X_o$, and $X_u$, as illustrated in Figure \ref{fig:rlou}, such that changing one of these crossings deforms $K$ into the trivial knot.
Moreover, $K$ can be deformed into the trivial knot by a single $\Delta$-move on $\Delta _r$ or $\Delta _l$. 

\end{example}
\begin{figure}[htbp]
    \centering
    \includegraphics[width=0.6\linewidth]{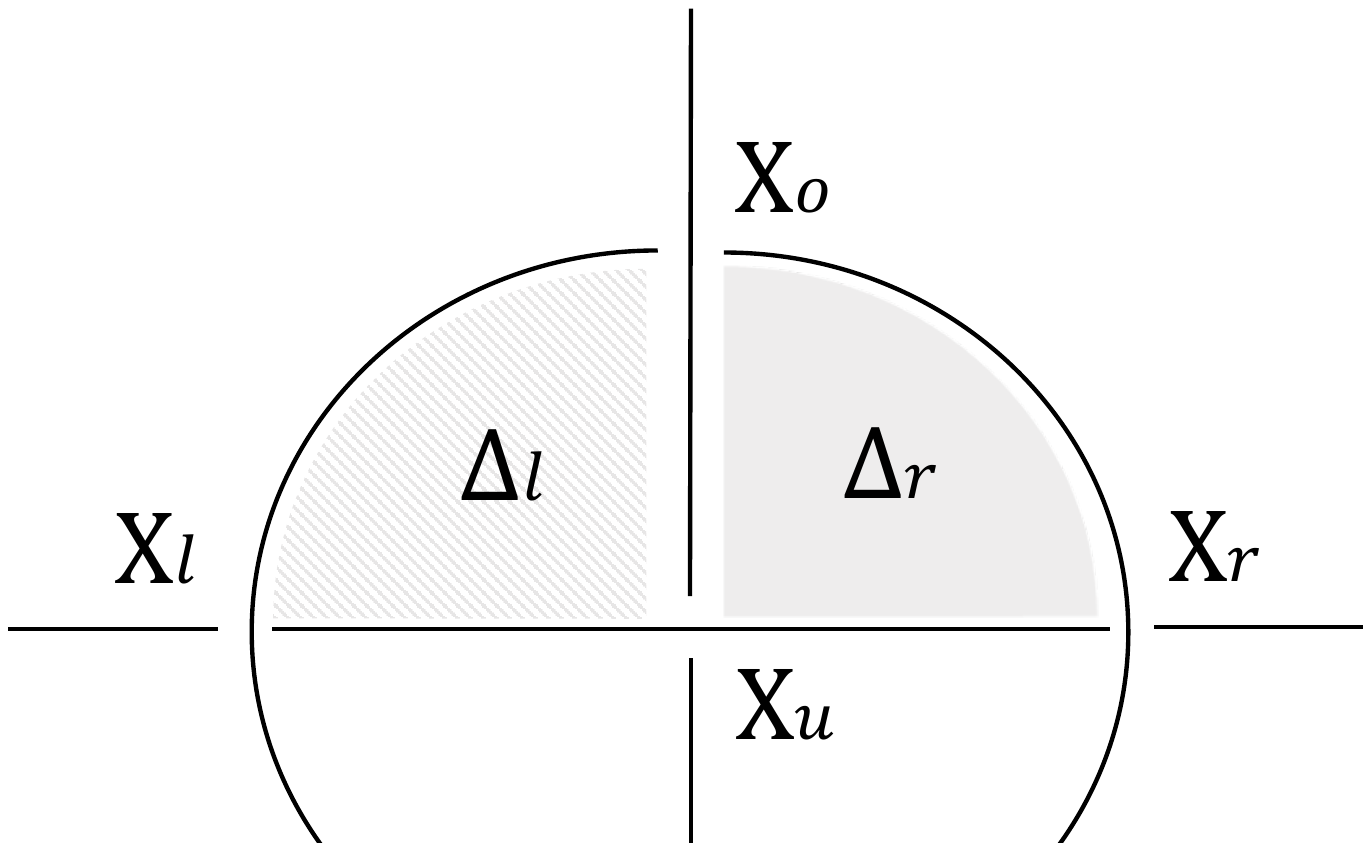}
    \caption{${\Delta}_r$, ${\Delta}_l$, $X_r$, $X_l$, $X_o$, and $X_u$.}
    \label{fig:rlou}
\end{figure}

\begin{table}[htbp]
\centering
\caption{Knots with $u^{\Delta}(K)=u(K)=1$ and at most 10 crossings.}
\label{tab:examples}
\begin{tabular}{c|l}
\hline
& Knots \\
\hline
Group A &
$3_1$, $4_1$, $6_2$, $6_3$, $7_6$, $7_7$, $8_{11}$, $8_{13}$, \\
(Example \ref{remark:oneone}) &
$9_{12}$, $9_{14}$, $9_{22}$, $9_{24}$, $9_{28}$, $9_{30}$, \\
&
$10_7$, $10_{10}$, $10_{59}$, $10_{60}$, $10_{71}$, $10_{73}$ \\
\hline
Group B &
$8_{17}$, $9_{33}$, $9_{34}$, $10_{32}$, $10_{88}$, $10_{104}$, \\
(Example \ref{remark:oneone2}) &
$10_{107}$, $10_{114}$, $10_{119}$, $10_{165}$ \\
\hline
Group C
&
$10_{141}$, $10_{147}$, $10_{156}$ \\
\hline
\end{tabular}
\end{table}

\vspace{1em}
\subsection{$\Delta$-Unknotting Number One and Unknotting Number Two} 
\label{subsec42}

\leavevmode

\vspace{1em}

It is known that a single $\Delta$-move can be realized by two $X$-moves.
Therefore, we obtain the following remark.

\begin{remark}\label{remark:onetwoall}
Let $K$ be a knot with $u^{\Delta}(K) = 1$ and $u(K) = 2$.
Then there exists a diagram of $K$ containing a single $\Delta$-configuration, as illustrated in Figure \ref{fig:deltaxx}, such that a single $\Delta$-move deforms $K$ into the trivial knot.
Moreover, $K$ can be deformed into the trivial knot by the two $X$-moves illustrated in Figure \ref{fig:deltaxx}. 
\end{remark}

\begin{figure}[htbp]
    \centering
    \includegraphics[width=0.6\linewidth]{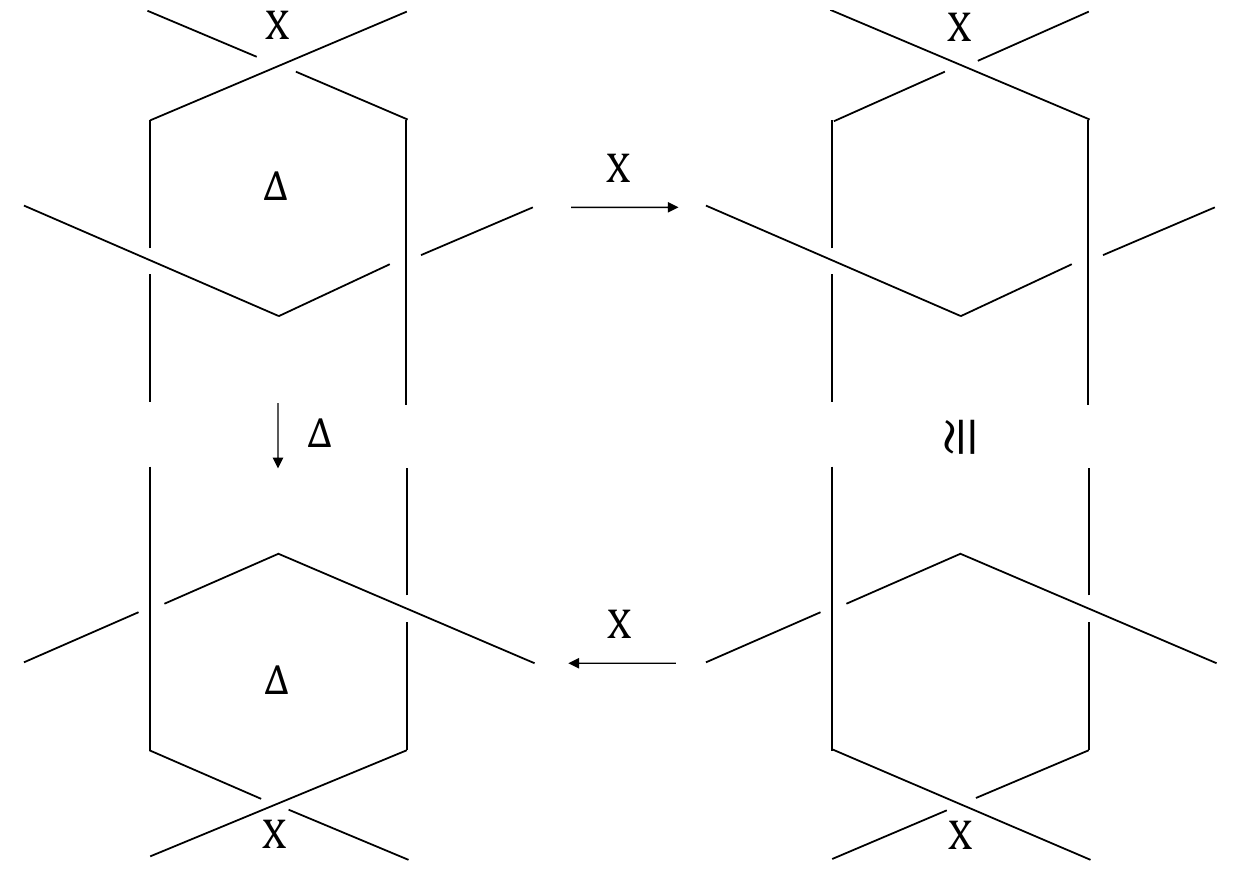}
    \caption{A realization of a $\Delta$-move by two $X$-moves.}
    \label{fig:deltaxx}
\end{figure}

On the other hand, the following examples show that every knot with at most 10 crossings admits a diagram in which both $X$-move sites are visible simultaneously.

\begin{example}\label{remark:onetwo}
Let $K$ be a knot with $u^{\Delta}(K) = 1$, $u(K) = 2$, and at most 10 crossings.
Then there exists a diagram of $K$ containing two crossings $X_r$ and $X_l$, as illustrated in Figure \ref{fig:rlouto}, such that changing both crossings deforms $K$ into the trivial knot.
Moreover, a single $\Delta$-move at either of the two $\Delta$-configurations $\Delta _r$ and $\Delta _l$ deforms $K$ into the trivial knot. 

Furthermore, let $K$ be a two-bridge knot as described in Theorem \ref{thm:deltaone}, and let $u(K) = 2$.  Then $K$ also has this property
(see Figure \ref{fig:tworlrl}). 

\end{example}

\begin{figure}[htbp]
    \centering
    \includegraphics[width=0.8\linewidth]{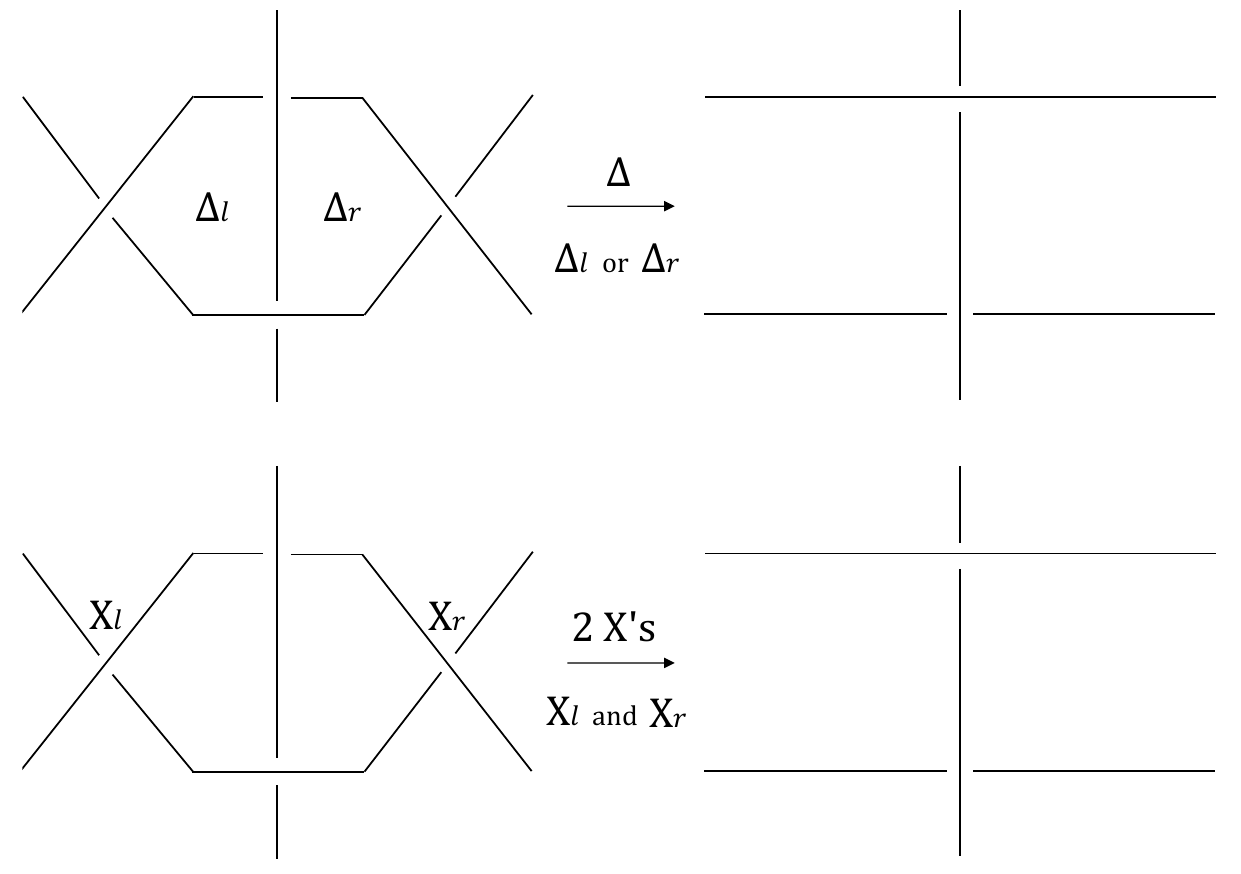}
    \caption{${\Delta}$-moves and $X$-moves.}
    \label{fig:rlouto}
\end{figure}

\begin{table}[htbp]
\centering
\caption{Knots with $u^{\Delta}(K)=1$, $u(K)=2$ and at most 10 crossings.}
\label{tab:groupd}
\begin{tabular}{c|l}
\hline
& Knots \\
\hline
Group D &
$8_{16}$, $8_{18}$, $9_{32}$, 
$9_{40}$, $9_{47}$, $10_{19}$, $10_{83}$,
$10_{86}$, \\
(Example \ref{remark:onetwo}) &
$10_{89}$, $10_{105}$, $10_{106}$, 
$10_{115}$, $10_{121}$, $10_{164}$ \\
\hline
\end{tabular}
\end{table}

\begin{figure}[htbp]
    \centering
    \includegraphics[width=1\linewidth]{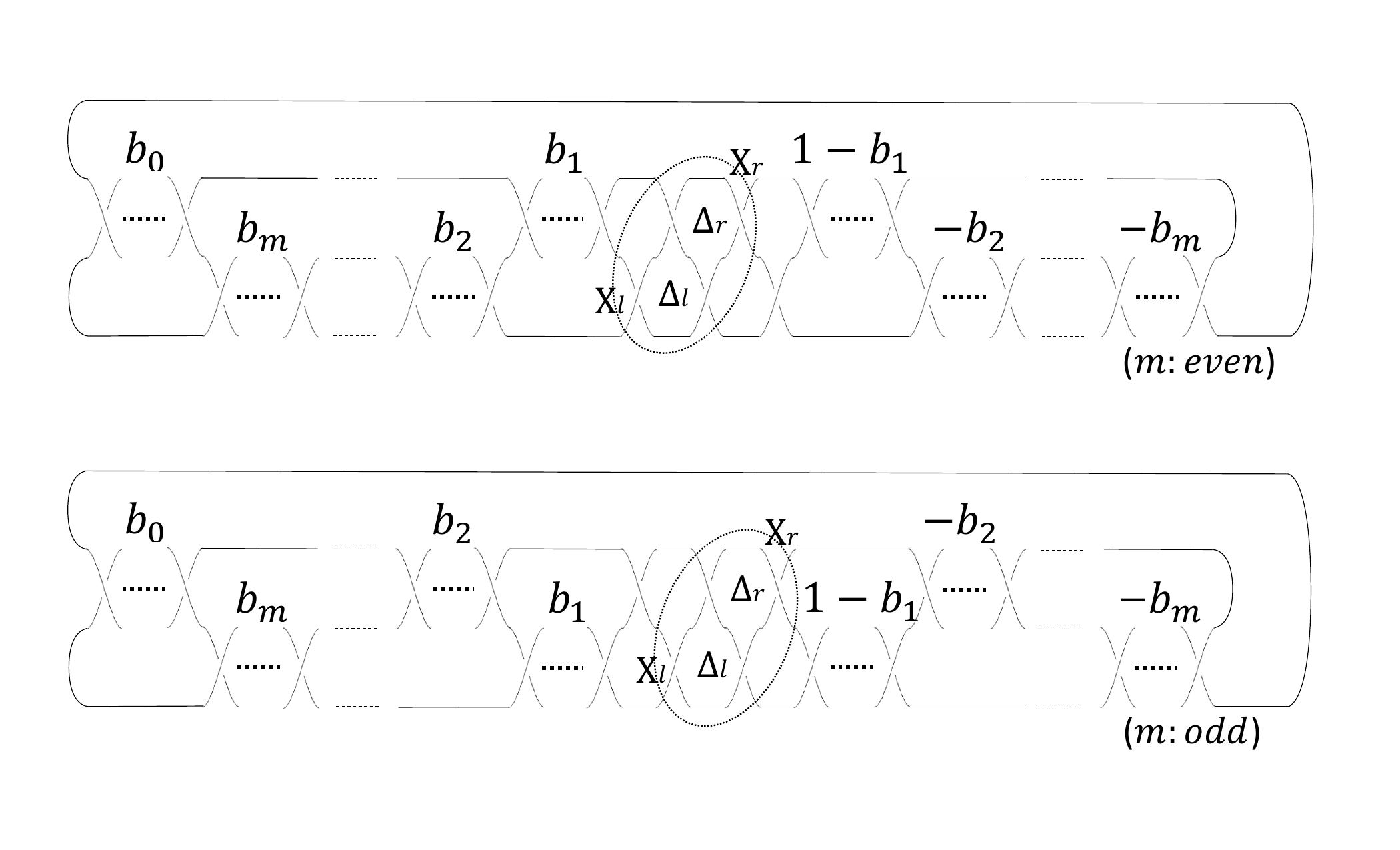}
    \caption{${\Delta}$-moves and $X$-moves in a diagram of the two-bridge knot $C(b_0, b_m, ... , b_2, b_1, 1, 1, 1, 1, 1, 1-b_1, -b_2, ... , -b_m)$.}
    \label{fig:tworlrl}
\end{figure}

\section*{Acknowledgments}
The author would like to express his sincere gratitude to Professor Makoto Sakuma for his valuable advice and continuous support.

The author is also grateful to Professor Yoshiaki Uchida for posing Problem~\ref{problemdelta}, which motivated this work.

Finally, the author would like to thank his family for their unwavering support and understanding throughout this work.





\end{document}